\documentclass[11pt]{article}

\usepackage{amsmath,amsthm}

\usepackage{amssymb,latexsym}

\usepackage{enumerate}

\newcommand{\MM}{{\cal M}}

\newcommand{\BR}{{\mathbb R}}

\newcommand{\dyw}{\mbox{\rm div}}

\newtheorem{theorem}{\bf Theorem}[section]
\newtheorem{proposition}[theorem]{\bf Proposition}
\newtheorem{lemma}[theorem]{\bf Lemma}

\theoremstyle{definition}
\newtheorem*{definition}{Definition}

\newtheorem{remark}[theorem]{Remark}

\numberwithin{equation}{section}

\begin{document}

\title {On  the structure of diffuse measures for parabolic capacities}
\author {Tomasz Klimsiak and Andrzej Rozkosz\\
{\small Faculty of Mathematics and Computer Science,
Nicolaus Copernicus University} \\
{\small  Chopina 12/18, 87--100 Toru\'n, Poland}\\
{\small E-mail addresses: tomas@mat.umk.pl (T. Klimsiak), rozkosz@mat.umk.pl (A. Rozkosz)}}
\date{}
\maketitle
\begin{center}
{\bf Abstract}
\end{center}
{\small
Let $Q=(0,T)\times\Omega$, where $\Omega$ is a bounded open subset of $\BR^d$.
We consider the parabolic $p$-capacity on  $Q$
naturally associated with the usual $p$-Laplacian.
Droniou, Porretta and Prignet have shown that if a bounded Radon
measure $\mu$ on  $Q$ is diffuse, i.e. charges no set of zero $p$-capacity, $p>1$, then
it is of the form
$\mu=f+\mbox{div}(G)+g_t$
for some $f\in L^1(Q)$, $G\in (L^{p'}(Q))^d$ and $g\in
L^p(0,T;W^{1,p}_0(\Omega)\cap L^2(\Omega))$. We show the converse of this result: if $p>1$, then each bounded Radon measure
$\mu$ on $Q$ admitting such a decomposition is diffuse. }





\section{Introduction}
\label{sec1}

Let $\Omega$ be a bounded open set  in $\BR^d$ and
$Q=(0,T)\times\Omega$ for some $T>0$. For $p>1$, the
parabolic $p$-capacity of an open subset $U$ of $Q$ is defined by  (see \cite{DPP,Pi})
\[
\mbox{cap}_p(U)=\inf\{\|u\|_W: u\in W,\,u\ge \mathbf{1}_{U}\mbox{ a.e. in } Q\},
\]
where $W=\{u\in L^p(0,T;V):u_t\in L^{p'}(0,T;V')\}$,
$V=W^{1,p}_0(\Omega)\cap L^2(\Omega)$ and $V'$ is the dual of $V$;
we endow $V$ with the norm
$\|u\|_V=\|u\|_{W^{1,p}_0(\Omega)}+\|u\|_{L^2(\Omega)}$, and $W$
with the norm
$\|u\|_{W}=\|u_t\|_{L^{p'}(0,T;V')}+\|u\|_{L^{p}(0,T;V)}$. The
capacity $\mbox{cap}_p$ is then extended to arbitrary Borel subset
of $Q$ in the usual way.

Let $\MM_b(Q)$ denote the space of all (signed) bounded   Radon
measures on $Q$ equipped with the norm $\|\mu\|_{TV}=|\mu|(Q)$,
where $|\mu|$ stands for the variation of $\mu$.  We call
$\mu\in\MM_b(Q)$ diffuse if it charges no set of zero parabolic
$p$-capacity, i.e. if $\mu(B)=0$ for any Borel $B\subset Q$ such
that $\mbox{cap}_p(B)=0$. We denote by $\MM_{0,b}(Q)$  the subset
of $\MM_b(Q)$ consisting of all diffuse measures.  Droniou, Poretta
and Prignet \cite{DPP} have shown that for every $\mu\in\MM_{0,b}(Q)$ there exists
$f\in L^1(Q)$, $G=(G^1,\dots,G^d)$ with $G^i\in L^{p'}(Q)$,
$i=1,\dots,d$, and $g\in L^p(0,T;V)$ such that
\begin{equation}
\label{eq1.1}
\mu=f+\mbox{div}(G)+g_t.
\end{equation}
The decomposition
(\ref{eq1.1}) plays crucial role in the study of evolution
problems with measure data whose model example is
\begin{equation}
\label{eq1.2}
\begin{cases}
u_t-\Delta_pu+h(u)=\mu &\mbox{in }Q,\\
u=u_0 &\mbox{on }\{0\}\times\Omega,\\
u=0 &\mbox{on }(0,T)\times\partial\Omega,
\end{cases}
\end{equation}
where $\Delta_pu=\dyw(|\nabla u|^{p-2}\nabla u)$ is the usual
$p$-Laplace operator, $p>1$, $u_0\in L^1(\Omega)$ and
$h:\BR\rightarrow\BR$ (see \cite{DPP,Pe,PPP2}).

The decomposition
(\ref{eq1.1}) is a counterpart to the decomposition of diffuse
measures proved in the stationary case  by Boccardo, Gallou\"et
and Orsina \cite{BGO} (see also \cite{KR:BPAN} for an extension to
the Dirichlet forms setting). In the stationary case, each finite
Borel measure $\mu$ on $\Omega$ that charges no set of zero $p$-capacity
admits decomposition of the form
\begin{equation}
\label{eq1.3} \mu=f+\mbox{div}(G),
\end{equation}
where $f\in L^1(\Omega)$, $G=(G^1,\dots  G^d)$ with $G^i\in
L^{p'}(\Omega)$, $i=1,\dots,d$. The decomposition (\ref{eq1.3}) proved to be
important and useful in the study of elliptic equations with measure data
(see, e.g., \cite{BMP,DMOP,DPP,MP}).

In the stationary case it is also known  that  if $\mu$ is a
bounded Borel measure on $\Omega$ admitting decomposition
(\ref{eq1.3}), then it is diffuse  (see \cite{BGO} and also
\cite{KR:BPAN} for  a related result concerning the capacity associated
with a general Dirichlet operator). In the parabolic  setting only a
partial result in this direction is known. The difficulty is caused
by the term $g_t$ appearing in (\ref{eq1.1}). Petitta, Ponce
and Porretta \cite{PPP2} (see also \cite{PPP1}) have shown that
if $\mu\in\MM_b(Q)$ admits decomposition (\ref{eq1.1}) with $g$
having the additional property that $g\in L^{\infty}(Q)$, then
$\mu$ is indeed diffuse. The problem whether one can dispense with
this additional assumption was left  open. It is worth noting here
that not every diffuse measure can be written in the form
(\ref{eq1.1}) with bounded $g$ (see \cite{PPP1,PPP2}).

In this note we show that if $p>1$, then in the parabolic case the  situation is the same as in the stationary case, i.e. if $\mu\in\MM_{b}(Q)$ satisfies  (\ref{eq1.1}), then it is diffuse.

\section{Main result}

Define $V,V',W$ as in Section \ref{sec1}. We denote by
$\langle\cdot,\cdot\rangle$ the duality pairing  between $V'$ and
$V$, and by $\langle\langle\cdot,\cdot\rangle\rangle$ the duality
pairing between  the dual space $W'$ of $W$ and $W$.

We start with recalling decompositions of $\Phi\in W'$ and
$\mu\in\MM_{0,b}(Q)$ proved in \cite{DPP}.

\begin{lemma}
\label{th2.1} For every $\Phi\in W'$ there exist $h\in L^{p'}(0,T;L^2(\Omega))$,
$g\in L^{p}(0,T;V)$ $G=(G^1,\dots,G^d)$ with $G^i\in
L^{p'}(Q)$, $i=1,\dots,d$ such that for every $u\in W$,
\begin{equation}
\label{eq2.1}
\langle\langle\Phi,u\rangle\rangle=\int_Qhu\,dt\,dx
-\int_QG\nabla u\,dt\,dx-\int^T_0\langle u_t,g\rangle\,dt.
\end{equation}
\end{lemma}
\begin{proof}
See \cite[Lemma 2.24]{DPP}.
\end{proof}

If $\Phi\in W'$ satisfies (\ref{eq2.1}), then we write
\[
\Phi=h+\dyw G+g_t.
\]

\begin{theorem}
\label{th2.2} If $\mu\in\MM_{0,b}(Q)$, then there exists $f\in
L^1(Q)$,  $g\in L^p(0,T;V)$ and $G=(G^1,\dots,G^d)$ with $G^i\in
L^{p'}(Q)$, $i=1,\dots,d$, such that for every $\eta\in C^{\infty}_c([0,T]\times\Omega)$,
\begin{equation}
\label{eq2.2}
\int_Q\eta\,d\mu=\int_Qf\eta\,dt\,dx
-\int_QG\cdot\nabla\eta\,dt\,dx -\int^T_0\langle\eta_t,g\rangle\,dt.
\end{equation}
\end{theorem}
\begin{proof}
See \cite[Theorem 2.28]{DPP}.
\end{proof}


\begin{definition}
Let $\Phi\in W'$.  We say that  $w\in L^p(0,T;V)$ is a weak solution to
the  Cauchy-Dirichlet problem
\begin{equation}
\label{eq2.4} w_t-\Delta_pw=\Phi,\qquad w(0,\cdot)=0,\qquad
w=0\,\mbox{  on  }\, (0,T)\times\partial \Omega
\end{equation}
if
\[
-\int^T_0\langle\eta_t,w\rangle\,dt +\int_{Q}|\nabla w|^{p-2}\nabla w\,
\nabla\eta\,dt\,dx=\langle\langle \Phi,\eta\rangle\rangle
\]
for all $\eta\in W$ with $\eta(T,\cdot)=0$.
\end{definition}



In what follows, $\{j_n\}$ is a  family of symmetric mollifiers defined on $\BR\times\BR^d$.
For a given $\Phi\in W'$ and a given decomposition  (\ref{eq2.1}) with $h,G,g$ having compact supports in $Q$, we define (for sufficiently large $n\ge 1$) $\Phi_n\in W'$ by
\begin{align}
\label{eq2.n}
\langle\langle \Phi_n,u\rangle\rangle =\int_Q h_n u\,dt\,dx&-\int_Q G_n\nabla u\,dt\,dx
-\int^T_0\langle  g_n, u_t\rangle\,dt,\quad u\in W,
\end{align}
where $h_n=h* j_n$, $G_n= G*j_n$ and $g_n= g*j_n$.
\begin{proposition}
\label{prop2.3}
Let $\Phi\in W'$.
\begin{enumerate}[\rm(i)]
\item There exists a unique weak solution $w$  to {\rm{(\ref{eq2.4})}}.

\item Assume that $\Phi$ admits decomposition \mbox{\rm(\ref{eq2.1})}  with some
$h,G,g$ having compact supports in $Q$.
Let $\Phi_n$ be given by {\rm (\ref{eq2.n})}
and let $w_n$ be a weak solution to the  problem
\[
(w_n)_t-\Delta_pw_n=\Phi_n,\qquad w_n(0,\cdot)=0,\qquad
w_n=0\,\mbox{  on  }(0,T)\times\partial\Omega.
\]
Then  $w_n\rightarrow w$ in $L^p(0,T;V)$.
\end{enumerate}
\end{proposition}
\begin{proof}
Part (i) is proved in  \cite[Theorem 3.1]{DPP}. To prove (ii), we modify slightly the proof of
\cite[Theorem 3.1]{DPP}. By the definition of a weak solution and (\ref{eq2.n}), for sufficiently large $n\ge1$,
\[
-\int^T_0\langle\eta_t,w_n-g_n\rangle\,dt +\int_{Q}|\nabla w_n|^{p-2}\nabla w_n\,
\nabla\eta\,dt\,dx=\int_Qh_n \eta\,dt\,dx
+\int^T_0\langle\chi_n,\eta\rangle\,dt,
\]
for every $\eta\in C_c^\infty([0,T]\times D)$ such that $\eta(T)=0$.
From the above equality it follows that  $w_n-g_n\in W$ and, by a standard approximation argument, that
\begin{align*}
&-\int^t_0\langle\eta_s,w_n-g_n\rangle\,ds +(\eta(t),(w_n-g_n)(t))_{L^2(\Omega)}
+\int_0^t\int_{\Omega}|\nabla w_n|^{p-2}\nabla w_n\,
\nabla\eta\,ds\,dx\\
&\qquad\qquad=\int_0^t\int_\Omega h_n \eta\,ds\,dx
+\int^t_0\langle \chi_n,\eta\rangle\,ds,\quad t\in(0,T],
\end{align*}
for every $\eta\in W$.
Therefore from the proof of  \cite[Theorem 3.1]{DPP} (see the last two equations in \cite[page 131]{DPP}) and \cite[Lemma 5]{BG} it follows that  $\nabla w_n\rightarrow \nabla w$  in $L^p(Q)$ and $w_n\rightarrow w$ in $L^p(0,T;L^2(\Omega))$.
By this and  \cite[(3.6)]{DPP} (see also the comment following it),
the sequence $\{w_n-g_n\}$ is bounded in $W$. Therefore, by \cite[Corollary 4]{Si}
and uniqueness of the solution to (\ref{eq2.4}), $w_n-g_n\rightarrow u-g$ in $L^p(Q)$. Since $g_n\rightarrow g$ in $L^p(Q)$, it follows that $w_n\rightarrow w$ in $L^p(Q)$. By what has been proved, $w_n\rightarrow w$ in $L^p(0,T;V)$.
\end{proof}

Lemma \ref{prop2.1}  below is the key to proving our main result. To state and prove it, we need some more notation.

Since $\mbox{cap}_p$ is subadditive (see \cite[Proposition 2.8]{DPP}),
each  $\mu\in\MM_b(Q)$ has a unique decomposition (see \cite{FST})
of the form
\begin{equation}
\label{eq2.03}
\mu=\mu_d+\mu_c\,,
\end{equation}
where $\mu_d\in\MM_{0,b}(Q)$ (the diffuse part of $\mu$) and
$\mu_c\in\MM_b(Q)$ is concentrated on a set  of
zero $p$-capacity (the so-called concentrated part of $\mu$).
For  $\mu\in\MM_b(Q)$ with decomposition (\ref{eq2.03}), we set
\[
 \mu_n=\mu *j_n,\qquad \mu^n_d=\mu_d *j_n,\qquad \mu^n_c=\mu_c *j_n.
\]
We denote by $\omega(n,m)$ (resp. $\omega(n,\delta)$)  any quantity such that
\[
\lim_{m\rightarrow \infty}\limsup_{n\rightarrow \infty}|\omega(n,m)|=0\quad\mbox{(resp.}\quad \lim_{\delta\downarrow0}\limsup_{n\rightarrow \infty}|\omega(n,\delta)|=0).
\]
For $m>0$, we set $T_m(t)=((-m)\wedge t)\vee m$, $t\in\BR$.

Let $D$ be an open subset of $Q$. We denote by $ \MM_b(D)\cap W'$ the set of elements $\Phi\in W'$ for which  there exists $c>0$ such that $|\langle\langle\Phi,\eta\rangle\rangle|\le c\|\eta\|_\infty,\, \eta\in C_c^\infty(D)$.
For given $\Phi\in  \MM_b(D)\cap W'$, we denote by $\Phi^{\mbox{\tiny meas},D}\in \MM_b(Q)$ the unique  measure such that
\[
\langle\langle\Phi,\eta\rangle\rangle=\int_D\eta\,d\Phi^{\mbox{\tiny meas},D},\quad \eta\in C_c^\infty(D)
\]
(see the comments following \cite[Definition 2.22]{DPP}).

\begin{remark}
In the proof of Lemma \ref{prop2.1},  we will use  \cite[Lemma 5]{Pe}, which was proved in \cite{Pe}  under the  assumption  that $p>(2d+1)/(d+1)$.
A close inspection of the proof of \cite[Lemma 5]{Pe} reveals that this additional assumption on $p$ is unnecessary.
The reason is that  this assumption on $p$ is needed in \cite{Pe} to apply
\cite[Lemma 4]{Pe}. However, from \cite[Remark 2.3]{DPP} it follows that the assertion of \cite[Lemma 4]{Pe} holds true for any $p>1$.
\end{remark}

\begin{lemma}
\label{prop2.1}
Let $D$ be an open subset of $Q$ and $\Phi\in\MM_b(D)\cap W'$.
Assume that $\Phi$ admits decomposition \mbox{\rm(\ref{eq2.1})}  with some
$h,G,g$ having compact supports in $Q$ and by  $u_n\in L^p(0,T;V)$ denote a weak solution to the
problem
\begin{equation}
\label{eq2.6} (u_n)_t-\Delta_pu_n=\Phi_n,\qquad u_n(0,\cdot)=0,\qquad
u_n=0\,\mbox{  on  }\, (0,T)\times\partial\Omega
\end{equation}
with $\Phi_n$ defined by \mbox{\rm(\ref{eq2.n})}. Then for every $\eta\in C_c^\infty(D)$,
\begin{align}
\label{eq2.7}
&\lim_{m\rightarrow \infty}\limsup_{n\rightarrow \infty}I(n,m)
=\int_D\eta\,d(\Phi^{\mbox{\tiny meas},D})_c\,,
\end{align}
where
\[
I(n,m)=\frac{1}{m}\int_{\{m\le u_n\le 2m\}}|\nabla u_n|^p\eta\,dt\,dx
-\frac{1}{m}\int_{\{-2m\le u_n \le -m\}}|\nabla u_n|^p\eta\,dt\,dx.
\]
\end{lemma}
\begin{proof}
Set $\nu=\Phi^{\mbox{\tiny meas},D}$, $\nu_n=(\Phi_n)^{\mbox{\tiny meas},D}$ and $\theta_{m}(s)=\frac{1}{m}(T_{2m}(s)-T_m(s))$, $\theta=|\theta_m|$, $\psi(s)=\theta(s)-1$, $\Psi(t)=\int_0^t\psi(s)\,ds$, $\Theta(t)=\int_0^t\theta(s)\,ds$. We extend $\nu,\nu_n$ to measures on  $Q$ by putting $\nu(Q\setminus D)=\nu_n(Q\setminus D)=0$.
Observe that $|\nu_n|\ll dt\otimes dx$, so by a standard approximation argument, for all  $w\in W$ with compact support in $D$,
\[
\langle\langle\Phi_n,w\rangle\rangle=\int_Qw\,d\nu_n.
\]
Moreover, for every fixed $w\in W$ with compact support in $D$, there exists $N\ge 1$ such that
\begin{equation}
\label{eq2.star.conv}
\int_Qw\,d\nu_n=\int_Qw\,d(\nu*j_n),\quad n\ge N.
\end{equation}
Indeed, for sufficiently large $n\ge 1$,
\begin{align*}
\int_Qw\,d(\nu*j_n)&
=\int_Q(w*j_n)\,d\nu=\langle\langle\Phi,w*j_n\rangle\rangle\\
&=\int_Qh(w*j_n)\,dt\,dx-\int_QG(\nabla w_n*j_n)\,dt\,dx-\int_Q(w*j_n)_tg\,dt\,dx\\
&=\langle\langle\Phi_n,w\rangle\rangle=\int_Qw\,d\nu_n.
\end{align*}
Let $E\subset Q$ be a Borel set  such that $\mbox{cap}_p(E)=0$ and $\nu_c$ is concentrated on $E$. By regularity of the measure $\nu$ and \cite[Lemma 5]{Pe}, for every $\delta>0$ there exists a compact set $K_\delta\subset E$, an open set $U_\delta\subset D$ such that $K_\delta\subset U_\delta$, and  $\psi_\delta\in C_c^1(U_\delta)$ with  $0\le \psi_\delta\le 1$ such that
\begin{equation}
\label{eq2.8}
|\nu|(U_\delta\setminus K_\delta)\le\delta,\qquad \int_Q (1-\psi_\delta)\,d|\nu_c|\le \delta,
\end{equation}
\begin{equation}
\label{eq2.9}
\|(\psi_\delta)_t\|_{L^1(Q)+L^{p'}(0,T;W^{-1,p'}(\Omega))}+ \|\psi_\delta\|_{L^p(0,T;V)}\le \delta,
\end{equation}
\begin{equation}
\label{eq2.09}
\psi_\delta\rightarrow 0\quad\mbox{weakly${}^*$ in }L^\infty(Q)\mbox{ as } \delta\downarrow0.
\end{equation}
Let $\eta\in C_c^\infty(D)$. Taking $\psi(u_n)\psi_\delta\eta$ as a test function in (\ref{eq2.6}), we obtain
\begin{align*}
\int_Q\psi(u_n)\psi_\delta\eta\,d\nu_n&=\int_Q(u_n)_t\psi(u_n)\psi_\delta\eta\,dt\,dx\\
&\quad +\int_Q|\nabla u_n|^{p-2}\nabla u_n\nabla (\psi(u_n)\psi_\delta\eta)\,dt\,dx
=:I_1+I_2.
\end{align*}
Clearly
\begin{align*}
I_1=\int_Q(\Psi(u_n))_t\psi_\delta\eta\,dt\,dx
=-\int_Q\Psi(u_n)(\psi_\delta\eta)_t\,dt\,dx
&=-\int_Q\Psi(u_n)(\psi_\delta)_t\eta\,dt\,dx\\
&\quad-\int_Q\Psi(u_n)\psi_\delta\eta_t\,dt\,dx.
\end{align*}
Since $\Psi$ is continuous and bounded, it follows from Proposition \ref{prop2.3}
and (\ref{eq2.9}) that $I_1=\omega(n,\delta)$.
We have
\begin{align}
\label{eq2.14}
I_2&=\int_Q|\nabla u_n|^p\psi'(u_n)\psi_\delta\eta\,dt\,dx
+\int_Q|\nabla u_n|^{p-2}\nabla u_n\nabla\psi_\delta \psi(u_n)\eta\,dt\,dx\nonumber\\
&\quad+\int_Q|\nabla u_n|^{p-2}\nabla u_n\psi(u_n)\nabla \eta\psi_\delta\,dt\,dx.
\end{align}
Using Proposition \ref{prop2.3} and (\ref{eq2.09}) shows that $\int_Q|\nabla u_n|^p\psi'(u_n)\psi_\delta\eta\,dt\,dx=\omega(n,\delta)$. Applying H\"older's inequality,
Proposition \ref{prop2.3} and (\ref{eq2.9}) also shows that the last two integrals on the right hand-side of (\ref{eq2.14}) are quantities of the form $\omega(n,\delta)$. Hence   $I_2=\omega(n,\delta)$, and consequently
\begin{equation}
\label{eq2.11}
\int_Q\psi(u_n)\psi_\delta\eta\,d\nu_n=\omega(n,\delta).
\end{equation}
Since $K_\delta\subset E$, $\mbox{cap}_p(K_\delta)=0$. Therefore, by (\ref{eq2.8}), $|\nu_d|(U_\delta)=|\nu_d|(U_\delta\setminus K_\delta)\le \delta$.
We also have $|\int_Q\psi(u_n)\psi_\delta\eta\,d\nu^n_d|\le\|\eta\|_{\infty}\int_Q\psi_{\delta}d|\nu_d|^n$ with $|\nu_d|^n=|\nu_d|*j_n$,
which converges to $\|\eta\|_{\infty}\int_Q\psi_{\delta}d|\nu_d|$ as $n\rightarrow\infty$ since $|\nu_d|^n\rightarrow|\nu_d|$ locally weakly${}^*$.
Thus $\int_Q\psi(u_n)\psi_\delta\eta\,d\nu^n_d=\omega(n,\delta)$. By this, (\ref{eq2.star.conv})  and (\ref{eq2.11}),
\begin{equation}
\label{eq2.12}
\int_Q\psi(u_n)\psi_\delta\eta\,d\nu^n_c=\omega(n,\delta).
\end{equation}
Taking $\theta(u_n)\eta$ as a test function in (\ref{eq2.6}) we obtain
\begin{align}
\label{eq2.13}
\int_Q\theta(u_n)\eta\,d\nu_n&=\int_Q(u_n)_t\theta(u_n)\eta\,dt\,dx+\int_Q|\nabla u_n|^{p-2}\nabla u_n\nabla (\theta(u_n)\eta)\,dt\,dx\nonumber\\
&=\int_Q(\Theta(u_n))_t\eta\,dt\,dx +\int_Q|\nabla u_n|^{p}\theta'(u_n)\eta\,dt\,dx\nonumber\\
&\quad+\int_Q|\nabla u_n|^{p-2}\nabla u_n\theta(u_n)\nabla\eta\,dt\,dx.
\end{align}
By the definition of $\theta$,
\[
\int_Q|\nabla u_n|^{p}\theta'(u_n)\eta\,dt\,dx=I(n,m).
\]
We have
\[
\Big|\int_Q(\Theta(u_n))_t\eta\,dt\,dx\Big|
=\Big|\int_Q\Theta(u_n)\eta_t\,dt\,dx \Big|
\le \int_{\{|u_n|\ge m\}}|u_n||\eta_t|\,dt\,dx
\]
and
\[
\Big|\int_Q|\nabla u_n|^{p-2}\nabla u_n\theta(u_n)\nabla\eta\,dt\,dx\Big|
\le\int_{\{|u_n|\ge m\}} |\nabla u_n|^{p-1}|\nabla\eta|\,dt\,dx,
\]
so by Proposition \ref{prop2.3},
\[
\int_Q(\Theta(u_n))_t\eta\,dt\,dx +\int_Q|\nabla u_n|^{p-2}\nabla u_n\theta(u_n)\nabla\eta\,dt\,dx=\omega(n,m).
\]
By the above and (\ref{eq2.13}),
\begin{equation}
\label{eq2.17}
I(n,m)=\int_Q\theta(u_n)\eta\,d\nu_n+\omega(n,m).
\end{equation}
By \cite[Theorem 1.2, Proposition 3.3]{PPP2},
\begin{equation}
\label{eq2.20}
\Big|\int_Q\theta(u_n)\eta\,d\nu^n_d\Big|\le \|\eta\|_\infty\int_{\{|u_n|\ge m\}}\,d|\nu_d|^n=\omega(n,m).
\end{equation}
Furthermore, by the definition of $\psi$,
\begin{equation}
\label{eq2.21}
\int_Q\theta(u_n)\eta\,d\nu^n_c=\int_Q\eta\,d\nu^n_c+\int_Q\psi(u_n)\eta\,d\nu^n_c,
\end{equation}
and by  (\ref{eq2.8})  and (\ref{eq2.12}),
\begin{equation}
\label{eq2.22}
\int_Q\psi(u_n)\eta\,d\nu^n_c
=\int_Q\psi(u_n)\eta(1-\psi_\delta)\,d\nu^n_c +\int_Q\psi(u_n)\eta\psi_\delta\,d\nu^n_c
=\omega(n,\delta).
\end{equation}
Since $\int_Q\theta(u_n)\eta\,d\nu_n$ does not depend on $\delta$, from (\ref{eq2.star.conv}) and  (\ref{eq2.20})--(\ref{eq2.22}) we conclude that
\[
\int_Q\theta(u_n)\eta\,d\nu_n=\int_Q\eta\,d\nu^n_c+\omega(n,m).
\]
Combining this with (\ref{eq2.17}) we see that
\[
I(n,m)=\int_Q\eta\,d\nu^n_c+\omega(n,m),
\]
which implies (\ref{eq2.7}).
\end{proof}

In case  $\Phi$  is positive, Lemma \ref{prop2.1} is essentially
\cite[Proposition 5]{PP}. Note that \cite[Proposition 5]{PP} is proved for any positive $\Phi\in\MM_b(Q)$. In Lemma \ref{prop2.1} we drop the assumption that  $\Phi$ is positive, but we additionally assume that $\Phi\in W'$.

\begin{theorem}
\label{th2.4}
Let  $\mu\in\MM_b(Q)$. If \mbox{\rm(\ref{eq2.2})} is satisfied
for all $\eta\in C^\infty_c(Q)$, then  $\mu\in\MM_{0,b}(Q)$.
\end{theorem}
\begin{proof}
Let $\nu=\mu-f\,dt\,dx$ and $\Phi= \dyw (G)+g_t$, i.e.
\[
\langle\langle\Phi,\eta\rangle\rangle
=-\int_Q G\nabla\eta\,dt\,dx-\int_0^T\langle \eta_t,g\rangle \,dt,\quad \eta\in W.
\]
Clearly $\Phi\in W'$. By (\ref{eq2.2}),
$\langle\langle\Phi,\eta\rangle\rangle =\int_Q\eta\,d\mu-\int_Q\eta f\,dt\,dx$ for
$\eta\in C^{\infty}_c(Q)$. From this and the assumption that $\mu\in\MM_b(Q)$ it follows
that $\Phi\in\MM_b(Q)\cap W'$ and $\Phi^{\mbox{\tiny meas},Q}=\nu$.
Fix an open subset of $Q$ such that $\bar D\subset Q$ and choose a nonnegative function $\theta\in C_c^\infty(Q)$ such that $\theta=1$ on $D$. Set  $G^\theta= G\theta$, $g^\theta= g\theta$, and then $\Phi^\theta=\dyw(G^\theta)+(g^\theta)_t$, i.e.
\[
\langle\langle\Phi^\theta,\eta\rangle\rangle
=-\int_Q G^\theta \nabla\eta\,dt\,dx-\int_0^T\langle\eta_t,g^\theta\rangle \,dt,
\quad \eta\in W.
\]
Next, set $G^\theta_n=G^\theta*j_n$,  $g^\theta_n=g^\theta*j_n$, and then
$\Phi^\theta_n=\dyw(G^\theta_n)+(g^\theta_n)_t$, i.e.
\begin{equation}
\label{eq2.24}
\langle\langle\Phi^\theta_n,\eta\rangle\rangle
=-\int_Q G^\theta_n \nabla\eta\,dt\,dx-\int_0^T\langle \eta_t,g^\theta_n\rangle \,dt,\quad \eta\in W.
\end{equation}
Clearly $\Phi^\theta,\Phi^\theta_n\in W'$. Since $\theta=1$ on $D$, we have
$\langle\langle\Phi^{\theta},\eta\rangle\rangle=\langle\langle\Phi,\eta\rangle\rangle$ for
$\eta\in C^{\infty}_c(D)$, so $\Phi^{\theta}\in\MM_b(D)\cap W'$. Integrating by parts, we conclude from (\ref{eq2.24}) that $\Phi^\theta_n\in \MM_b(D)\cap W'$.
Moreover,
\begin{equation}
\label{eq2.23}
(\Phi^\theta)^{\mbox{\tiny meas},D}=\nu_{|D},
\end{equation}
where $\nu_{|D}$ denotes the restriction of $\nu$ to $D$. Indeed, for  $\eta\in C^{\infty}_c(D)$ we have $\langle\langle\Phi^{\theta},\eta\rangle\rangle=\int_D\eta\,d(\Phi^\theta)^{\mbox{\tiny meas},D}$,
and on the other hand,  $\langle\langle\Phi^{\theta},\eta\rangle\rangle=\langle\langle\Phi,\eta\rangle\rangle
=\int_D\eta\,d\nu=\int_D\eta\,d\nu_{|D}$.
Let $u_n$ be a solution to (\ref{eq2.6}) with $\Phi_n$ replaced by $\Phi^\theta_n$.  From Proposition \ref{prop2.3} it follows that
$\sup_{n\ge 1}\| u_n\|_{L^p(0,T;V)}<\infty$.
Hence, for every $\eta\in C_c^\infty(Q)$,
\[
\lim_{m\rightarrow \infty}\limsup_{n\rightarrow \infty}\frac{1}{m}
\int_{\{m\le |u_n|\le 2m\}}|\nabla u_n|^p\eta\,dt\,dx\le
\|\eta\|_\infty\lim_{m\rightarrow \infty}\limsup_{n\rightarrow \infty}\frac{1}{m}\|u_n\|^p_{L^p(0,T;V)}=0.
\]
By Lemma \ref{prop2.1} and (\ref{eq2.23}),  this implies that $(\nu_{|D})_c=0$. Hence $(\mu_c)_{|D}=(\mu_{|D})_c=0$ since $f\,dt\,dx\in \MM_{0,b}(Q)$.
Since $D$ was an arbitrary open subset of $Q$ with $\bar D\subset Q$, we see that $\mu_c=0$.
\end{proof}

\subsection*{Acknowledgements}

{\small This work was supported by Polish National Science Centre
(Grant No. 2016/23/B/ST1/01543).}

\end{document}